\documentclass[11pt]{article}

\usepackage{latexsym,color,graphicx,amsmath,amssymb,amsthm}

\def\reals{{\mathbb R}}
\def\R{{\mathbb R}}
\def\C{{\mathbb C}}

\def\cl{{\rm Cl}}

\newtheorem{theorem}{Theorem}
\newtheorem{lemma}[theorem]{Lemma}
\newtheorem{corollary}[theorem]{Corollary}

\def\reals{{\mathbb R}}
\def\R{{\mathbb R}}

\def\C{{\mathbb C}}

\def\S{{\mathbb S}}

\def\cl{{\rm Cl}}

\def\minus{\backslash}
\def\p{{\bf p}}

\def\q{{\bf q}}

\theoremstyle{definition}
\newtheorem{definition}[theorem]{Definition}

\begin{document}

\title{Configurations of lines in space and combinatorial rigidity\footnote{Work on this paper 
was supported by Grant 892/13 from the Israel Science Foundation,
by the Israeli Centers of Research Excellence (I-CORE) program (Center No.~4/11),
and by a Shulamit Aloni Fellowship from the Israeli Ministry of Science.
}}

\author{Orit E. Raz\thanks{School of Computer Science, 
Tel Aviv University, Tel Aviv 69978, Israel. oritraz@post.tau.ac.il
}}
\maketitle

\begin{abstract}
Let $L$ be a sequence $(\ell_1,\ell_2,\ldots,\ell_n)$ of $n$ lines in $\C^3$.  
We define the {\it intersection graph} $G_L=([n],E)$ of $L$, where $[n]:=\{1,\ldots, n\}$, and with $\{i,j\}\in E$ 
if and only if $i\neq j$ and the corresponding lines $\ell_i$ and $\ell_j$ intersect,
or are parallel (or coincide). For a graph $G=([n],E)$, we say that a sequence $L$ 
is a {\it realization} of $G$ if $G\subset G_L$.
One of the main results of this paper is to provide a combinatorial characterization of graphs $G=([n],E)$ that
have the following property: 	
For every {\it generic} (see Definition~\ref{def:generic}) realization $L$ of $G$, 
that consists of $n$ pairwise distinct lines, we have $G_L=K_n$,
in which case the lines of $L$ are either all concurrent or all coplanar. 

The general statements that we obtain about lines, apart from their independent interest, turns out to be 
closely related to the notion of graph rigidity.
The connection is established due to the so-called Elekes--Sharir framework, which allows us to transform
the problem into an incidence problem involving lines in three dimensions.
By exploiting the geometry of contacts between lines in 3D, we can
obtain alternative, simpler, and more precise characterizations of the rigidity 
of graphs.
\end{abstract}

\section{Introduction}

Let $L$ be a sequence $(\ell_1,\ell_2,\ldots,\ell_n)$ of $n$ lines in $\C^3$.  
We define the {\it intersection graph} $G_L=([n],E)$ of $L$, where $[n]:=\{1,\ldots, n\}$, and with $\{i,j\}\in E$ 
if and only if $i\neq j$ and the corresponding lines $\ell_i$ and $\ell_j$ intersect,
or are parallel (or coincide). For a graph $G=([n],E)$, we say that a sequence $L$ 
is a {\it realization} of $G$ if\footnote{For 
graphs $G_1=(V_1,E_1)$ and $G_2=(V_2,E_2)$, we say that $G_1\subset G_2$ if $V_1=V_2$ and $E_1\subseteq E_2$.}
$G\subset G_L$.
 
We consider the following general question: What can be said about realizations $L$ of a certain graph $G$?  
In particular, we are interested in conditions on graphs $G$
that guarantee that, for every realization $L$ of $G$,
the lines of $L$ 
must be either all concurrent or all coplanar.
In other words, we want conditions on $G$ that guarantee that, for every realization $L$ of $G$,
we have $G_L=K_n$, the complete graph on $n$ vertices.
(It can be easily verified that $G_L=K_n$ implies that the lines of $L$ 
must be either all concurrent or all coplanar.)
Unfortunately, already by removing one edge of $K_n$,
this property seems to fail: One can easily find configurations $L$ with lines that are neither all concurrent nor all coplanar, and such that 
$G_L=K_n\setminus \{1,2\}$, say. Indeed, consider $n-2$ lines $\ell_3,\ldots,\ell_n$ that are all concurrent {\it and} all coplanar,
let $\ell_1$ be any line that lies on the common plane supporting those lines 
(but does not go through their common intersection point), and let 
$\ell_2$ be any line that goes through the common intersection point of $\ell_3,\ldots,\ell_n$ (but does not lie on
the plane supporting those lines). Then for $L=(\ell_1,\ell_2,\ldots,\ell_n)$, we get $G_L=K_n\setminus \{1,2\}$,
and the lines of $L$ are neither all concurrent nor all coplanar. 
Note however that in this example we had to use $n-2$
lines which are both concurrent and coplanar, which we would like to think of as a degenerate configuration of lines. 

In this paper we characterize graphs $G=([n],E)$ with the property that, for every {\it generic} (see Definition~\ref{def:generic})
realization $L$ of $G$, 
that consists of $n$ pairwise distinct lines, we have $G_L=K_n$ 
(that is, the lines of $L$ are either all concurrent or all coplanar). 

In the background of our results lies a connection (that we establish here) between line configurations
and the classical notion of rigidity of planar realizations of graphs (see Section~\ref{sec:rigid} for the definitions).
The connection is established (in Section~\ref{reduction}) due to the so-called Elekes--Sharir framework, which allows us to transform
the problem into an incidence problem involving lines in three dimensions.
By exploiting the geometry of contacts between lines in 3D, we can
obtain alternative, simpler, and more precise characterizations of the rigidity 
of graphs.

\section{Intersection graph and the variety $X_G$}

For simplicity, in this paper all the lines are assumed to be non-horizontal. 
Let $L$ be a sequence $(\ell_1,\ell_2,\ldots,\ell_n)$ of $n$ (not necessarily distinct) complex lines in $\C^3$.  
A line $\ell$ can be parametrized as
$$
\ell(t)=(a,b,0)+t(c,d,1),\quad t\in\C,
$$
for certain unique $a,b,c,d\in\C$, and we may represent $\ell$ as a point $u=(a,b,c,d)\in \C^4$, in this sense.
By identifying $\C^{4n}$ with $(\C^4)^n$, we may regard a sequence $L=(\ell_1,\ell_2,\ldots,\ell_n)$ of $n$ lines in $\C^3$ 
as a point $(u_1,\ldots,u_n)\in\C^{4n}$, where each $u_i=(a_i,b_i,c_i,d_i)\in\C^4$ 
is the point representing the line $\ell_i$ of $L$.
Similarly, every point $x\in\C^{4n}$ can be interpreted as a sequence $L=L(x)$ of $n$ (not necessarily distinct) 
lines in $\C^3$.

Given a graph $G=([n],E)$, we define the variety $X_G:=\cl (\hat X_G)$, where
$\cl(S)$ is the {\it Zariski closure} of a set $S\subset\C^{4n}$, and
$$
\hat X_G:=\left\{x\in \C^{4n}\mid G\subset G_{L(x)}~\text{and the lines of $L(x)$ are pairwise distinct}\right\}.
$$

Note that for every point $x\in X_G$, we have $G\subset G_{L(x)}$ (where 
here it is possible for some of the lines of $L(x)$ to coincide).
Indeed, we have $\hat X_G\subset Y_G$, where
$Y_G:=\left\{x\in \C^{4n}\mid G\subset G_{L(x)}\right\}$.
Since $Y_G$ is an algebraic variety, it follows that $X_G\subset Y_G$.
To see that $Y_G$ is a variety, 
let $g$ be the polynomial in the eight coordinates $a_1$, $b_1$, $c_1$, $d_1$, $a_2$, $b_2$, $c_2$, $d_2$,
which vanishes if and only if the pair of lines associated with the coordinates $(a_1,b_1,c_1,d_1)$ and 
$(a_2,b_2,c_2,d_2)$ intersect or are parallel.
Namely, 
\begin{equation}\label{intersect}
g(a_1,b_1,c_1,d_1,a_2,b_2,c_2,d_2)=(a_1-a_2)(d_1-d_2)-(b_1-b_2)(c_1-c_2).
\end{equation}
By definition, $Y_G$ is given by the system
$$
g(a_i,b_i,c_i,d_i,a_j,b_j,c_j,d_j)=0~~~~\forall \;(i,j)\in E,
$$
and is therefore a variety, as claimed.

We have the following simple observation.
\begin{lemma}\label{complete}
The variety $X_{K_n}$ is $(2n+3)$-dimensional.
\end{lemma}
\begin{proof}
First note that a point in $\hat X_{K_n}$ corresponds to a sequence $L$ of $n$ mutually intersecting lines, 
and hence its lines are either all concurrent (or mutually parallel), or all coplanar. 

A line, given by the parameterization
$$
\ell(t)=(a,b,0)+t(c,d,1),~~~t\in\C
$$
passes through a point $(x,y,z)\in \C^3$ if and only if
$a=x-cz$ and $b=y-dz$. Thus, sequences $L$ of $n$ concurrent lines in $\C^3$ 
form a Zariski-dense subset of the following variety
$$
\{(x-c_1z,y-d_1z,c_1,d_1,\ldots, x-c_nz,y-d_nz,c_n,d_n)\mid x,y,z,c_1,d_1,\ldots,c_n,d_n\in\C\},
$$
which is $(2n+3)$-dimensional, as it is the image of $\C^{2n+3}$ under a polynomial function. 

Sequences $L$ of $n$ mutually parallel lines in $\C^3$ span the 
$(2n+2)$-dimensional subvariety of $X_{K_n}$ given by 
$$\{(a_1,b_1,c_0,d_0,\ldots, a_n,b_n,c_0,d_0)\in \C^{4n}\mid c_0,d_0,a_1,b_1,\ldots,a_n,b_n\in\C\}.
$$
Similarly, a line $\ell(t)=(a,b,0)+t(c,d,1)$ lies in a plane given by $z=\lambda x+\mu y+\nu$, if and only if 
$1=\lambda c+\mu d$ and $-\nu=\lambda a+\mu b$. Thus,
sequences $L$ of $n$ coplanar lines form a Zariski-dense subset of the following variety 
$$
\{(a_1,b_1,c_1,d_1,\ldots, a_n,b_n,c_n,d_n)\mid 1=\lambda c_i+\mu d_i, -\nu= \lambda a_i+\nu b_i, \lambda,\mu,\nu\in\C, i\in[n] \}.
$$
which is again $(2n+3)$-dimensional.
So $X_{K_n}$ is the union of these three varieties, and hence it is $(2n+3)$-dimensional. 
\end{proof}

The following simple lemma asserts that $X_G$ is always at least $(2n+3)$-dimensional, 
where $n$ is the number of vertices of $G$.
Later on (in Corollary~\ref{cor:main}), we introduce a sufficient and necessary 
condition on $G$ for $X_G$ to be exactly $(2n+3)$-dimensional.
\begin{lemma}\label{lower}
Let $G$ be a graph on the vertex set $[n]$. Then $X_G$ is of dimension at least $2n+3$.
\end{lemma}
\begin{proof}
Since $G$ (and every graph on $[n]$) is contained in $K_n$, we have, by definition, that $\hat X_{K_n}\subset \hat X_G$ and hence also
$X_{K_n}\subset X_G$.
By Lemma~\ref{complete}, $X_{K_n}$ is $(2n+3)$-dimensional, and thus $X_G$ is of dimension at least $2n+3$, as asserted.
\end{proof}

\section{Laman graphs}

We recall the definition of {\it Laman graphs}.
These graphs play a fundamental role in the theory of {\it combinatorial rigidity} of graphs in the plane, and were discovered by Laman \cite{Lam70}.
In our definition we refer only to the combinatorial properties of those graphs.
Later,  in Section~\ref{reduction}, the ``reason'' why these graphs pop up also in our context will become more apparent.

\begin{definition}
 A graph $G = (V ,E)$ is called a {\it Laman graph} if\\
(i) $|E| = 2|V|-3$, and\\
(ii)  every subgraph $G'=(V',E')$ of $G$, with $|V'|\ge 2$, satisfies $|E'|\le 2|V'|-3$.
\end{definition}

The main result of this section is the following theorem.
\begin{theorem}\label{main}
If $G$ is a Laman graph on $n$ vertices, then $X_G$ is $(2n+3)$-dimensional.
\end{theorem}

It follows from Theorem~\ref{main} that, for every graph $G$ on $n$ vertices that contains a subgraph $G'\subset G$ which is Laman, 
we have $\dim X_G=2n+3$. 
Indeed, we have $G'\subset G\subset K_n$, which implies that $ X_{K_n}\subset X_G\subset X_{G'}$, and, in particular,
$\dim X_{K_n} \le \dim X_G\le \dim X_{G'}$. Combining Theorem~\ref{main} and Lemma~\ref{complete},
the claim follows. 
In Section~\ref{sec:nec} we show (in Theorem~\ref{mainnec}) that it is also necessary for $G$ to contain a subgraph 
which is Laman, in order for $X_G$ to be ($2n+3$)-dimensional. We thus obtain the following
characterization.
\begin{corollary}\label{cor:main}
Let $G$ be a graph on $n$ vertices.
Then $X_G$ is $(2n+3)$-dimensional
if and only if there exists a subgraph $G'\subset G$ which is Laman. 
\end{corollary}

For the proof of Theorem~\ref{main} we use the following constructive characterization
of Laman graphs (called the Henneberg or Henneberg--Laman construction; 
see Jackson and Jord\'an~\cite[Corollary 2.12]{JJ05} and references therein).
In the statement, a $0$-{\it extension} means adding a new vertex of degree $2$ to $G$, and a $1$-{\it extension} 
means subdividing an edge $uv$ of $G$ by a new vertex $z$, and adding a new edge $zw$ for some $w\neq u, v$.
That is, the edge set of the new graph is $(E\setminus \{uv\})\cup \{uz,vz,wz\}$.
\begin{theorem}[Henneberg--Laman]\label{constLaman}
A graph $G = (V ,E)$ is Laman if and only if $G$ can be obtained from $K_2$
by a sequence of $0$- and $1$-extensions.
\end{theorem}

We observe the following simple property.
\begin{lemma}\label{3lines}
Let $\ell_1,\ell_2,\ell_3$ be three distinct lines in $\C^3$. Consider the variety $W$ of lines $\ell\subset\C^3$ that intersect 
each of $\ell_i$, for $i=1,2,3$ (similar to above, $W$ can be interpreted as a 
variety in $\C^4$). If $\ell_1,\ell_2,\ell_3$ are either concurrent or coplanar, $W$ is two-dimensional.
Otherwise, $W$ is one-dimensional.
\end{lemma}
\begin{proof}
If $\ell_1,\ell_2, \ell_3$ are both concurrent and coplanar, then each $\ell\in W$ must either pass through their common point
or lie on their common plane.  If $\ell_1,\ell_2, \ell_3$ are concurrent and not coplanar, then 
each $\ell\in W$ must pass through their common point. If the they are coplanar but not concurrent, 
each $\ell\in W$ must lie on their common plane. In each of these scenarios, $\ell$ has two degrees of freedom.

If the three lines are pairwise skew, $\ell\in W$ must be a generator line of the regulus that they span, 
and then has only one degree of freedom.

It is left to handle the case where $\ell_1,\ell_2,\ell_3$ are neither concurrent nor coplanar, and a pair of them, say, $\ell_1,\ell_2$, is concurrent.
Then $\ell_3$ intersects their common plane in a single point $q$, which is not $\ell_1\cap \ell_2$.
Then each $\ell\in W$ must either pass through $\ell_1\cap\ell_2$ and an arbitrary point of $\ell_3$, or pass through $q$
and lie on the common plan spanned by $\ell_1,\ell_2$. In either case $\ell$ has one degree of freedom.
\end{proof}

\subsection{Proof of Theorem~\ref{main}}
By Lemma~\ref{lower}, $X_G$ is at least $(2n+3)$-dimensional. So it suffices to show that, in case $G$ is Laman, 
$X_G$ is at most $(2n+3)$-dimensional.

We use induction on the number $n\ge 2$ of vertices of $G$. 
For the base case $n=2$, the only graph to consider is $G=K_2$, for which $X_G$ is 7-dimensional, by Lemma~\ref{complete}.
Assume the correctness of the statement for each $2\le n'<n$ and let 
$G$ be a Laman graph on the vertex set $[n]$.
By Theorem~\ref{constLaman}, $G$ can  be obtained from $K_2$
by a sequence of $N=n-2$ extensions (note that each extension adds one new vertex to the graph).
Fix such a sequence, and let $G'$ be the graph obtained after applying the first $N-1$ 
extensions in the sequence. So $G$ can be obtained 
from $G'$ by applying a $0$- or $1$-extension to $G'$, and $G'$ is Laman.
Up to renaming the vertices, we may assume that the vertex set of $G'$ is $[n-1]$. 
By the induction hypothesis $\dim X_{G'}=2(n-1)+3=2n+1$.

Suppose first that $G$ is obtained from $G'$ by a $0$-extension, that is, by adding a new vertex, $n$, to $G'$ and two edges connecting 
$n$ to some pair of vertices, say, $1,2$, of $G'$. 

Let $\pi:\C^{4n}\to \C^{4n-4}$ denote that projection of $\C^{4n}$ onto its first $4n-4$ coordinates.
We claim that
$
\pi(\hat X_G)=\hat X_{G'}.
$
Indeed, for every sequence of lines $L=(\ell_1,\ldots,\ell_{n-1},\ell_n)$ with $G\subset G_L$, removing the line 
$\ell_n$ results in a sequence $L':=(\ell_1,\ldots,\ell_{n-1})$ with $G'\subset G_{L'}$; so  $\pi(\hat X_G)\subseteq \hat X_{G'}$.
Conversely, for every sequence $L'=(\ell_1,\ldots,\ell_{n-1})$ with $G'\subset G_{L'}$, there exists a line $\ell_n$ 
(in fact, a two-dimensional family of lines --- see below) that intersects both $\ell_1,\ell_2$, and thus, for 
$L:=(\ell_1,\ldots,\ell_{n-1},\ell_n)$, we have $G\subset G_L$; so also $\hat X_{G'}\subseteq \pi(\hat X_G)$.
This implies $\pi(\hat X_G)=\hat X_{G'}$, as claimed.

An easy property in algebraic geometry, given in Lemma~\ref{pre:proj} in the Appendix, implies that
$$
X_{G'}=\cl(\hat X_{G'})=\cl(\pi(\hat X_G))=\cl (\pi(X_G)).
$$
Note also that for every $y\in \hat X_{G'}$, we have that $\pi^{-1}(y)\cap X_G$ is two-dimensional.
Indeed, as was already noted, writing $L(y)=(\ell_1,\ldots,\ell_{n-1})$, 
we have $x\in \pi^{-1}(y)\cap X_G$
for every  $L(x)$ of the form
$L(x)=(\ell_1,\ldots,\ell_{n-1},\ell_n)$, where 
the first $n-1$ lines are fixed and the line $\ell_n$ is any line that intersects both $\ell_1,\ell_2$.
The set of such lines is two-dimensional.
Indeed, each such line can be parameterized by the pair of points of its intersection with $\ell_1$ and $\ell_2$,
except for lines that pass through the intersection point $\ell_1\cap \ell_2$, if such a point exists.
However, the space of lines of this latter kind is also two-dimensional, and the claim follows.

Let $X$ be any irreducible component of $\hat X_{G}$. Since $X_G=\cl(\hat X_G)$, we must
have $X\cap \hat X_G\neq \emptyset$, for every such component $X$ (otherwise, the union of the irreducible components of $X_G$,
excluding $X$, already contains $\hat X_G$). 
Let $y\in \pi(X\cap \hat X_G)\subset \hat{X}_{G'}$.
By another basic property in algebraic geometry, given in Lemma~\ref{pre:fib} in the Appendix, we have,
for each $y\in \hat X_{G'}$,
$$
\dim X
\le \dim \pi(X)+\dim (\pi^{-1}(y)\cap X)
\le \dim X_{G'}+\dim (\pi^{-1}(y)\cap X_G)\le 2n+3. 
$$
Since this holds for every irreducible component $X$ of $X_G$, we get $\dim X_G\le 2n+3$, 
which completes the proof of the theorem for this case.

Assume next that $G$ is obtained from $G'$ by applying a $1$-extension, that is, 
by subdividing an edge, say $\{i,j\}$ of $G'$ by a new vertex $n$, and adding a new
edge $\{n,k\}$ for some $k\neq i, j$. 
That is, $G$ is obtained from $G'$ by replacing the edge $\{i,j\}$ by the 
three edges $\{i,n\}$, $\{j,n\}$, and $\{k,n\}$.
Without loss of generality assume $i=1$, $j=2$, $k=3$.
Consider the graph $G'_{12}$ resulting by
removing the edge $\{1,2\}$ from $G'$, and the corresponding variety $X_{G'_{12}}$. 

Arguing similar to above, we have $X_{G_{12}'}=\cl (\pi(X_G))$.
Indeed, for every sequence of lines $L=(\ell_1,\ldots,\ell_{n-1},\ell_n)$ with $G\subset G_L$, removing the line 
$\ell_n$ results in a sequence $L':=(\ell_1,\ldots,\ell_{n-1})$ with $G_{12}'\subset G_{L'}$; so  $\pi(\hat X_G)\subseteq \hat X_{G'}$.
Conversely, for every sequence $L'=(\ell_1,\ldots,\ell_{n-1})$ with $G_{12}'\subset G_{L'}$, there exists a line $\ell_n$ 
(in fact, a family of lines which is at least one-dimensional, by Lemma~\ref{3lines}) 
that intersects each of $\ell_1,\ell_2,\ell_3$, and thus, for 
$L:=(\ell_1,\ldots,\ell_{n-1},\ell_n)$, we have $G\subset G_L$; so also $\hat X_{G'}\subseteq \pi(\hat X_G)$.
This implies $\pi(\hat X_G)=\hat X_{G'}$. Applying Lemma~\ref{pre:proj} in the Appendix, we get 
$X_{G_{12}'}=\cl (\pi(X_G))$, as claimed.

Let $g$ be the polynomial in \eqref{intersect} in the first eight coordinates of $\C^{4n-4}$, 
which vanishes if and only if the pair of lines $\ell_1,\ell_2$
associated with those 8 coordinates intersect.
Let $Z(g)$ denote the zero-set of $g$ in $\C^{4n-4}$.
Clearly, $X_{G'}= X_{G'_{12}}\cap Z(g)\subset X_{G'_{12}}$.

Assume, without loss of generality, that $X_{G'_{12}}$ is irreducible 
(otherwise apply the same argument to each of its irreducible components).
By Lemma~\ref{pre:dim} in the Appendix, we have either $X_{G'}=X_{G'_{12}}$ 
and thus $\dim X_{G'_{12}}=\dim X_{G'}\le 2n+1$,
 or, otherwise,  $\dim X_{G'}=\dim X_{G'_{12}}-1$.

Suppose that the former case occurs, that is, $\dim X_{G'_{12}}\le 2n+1$. 
It follows from Lemma~\ref{3lines} that
$\dim (\pi^{-1}(y)\cap X_G)\le 2$, for every $y\in \hat X_{G'_{12}}$. 
Lemma~\ref{pre:fib} then implies that
$$
\dim X_G\le  \dim X_{G_{12}'}+\dim (\pi^{-1}(y)\cap X_G)\le 2n+1+2=2n+3,
$$
as needed.

Suppose next that the latter case occurs, that is,  
$$
\dim X_{G'}=\dim (X_{G'_{12}}\cap Z(g))=\dim X_{G'_{12}}-1
$$
(and, in particular, $\dim X_{G'_{12}}\le 2n+2$).

Choose a point $y_0\in \hat X_{G'_{12}}\setminus Z(g)$, and write 
$L(y_0):=(\ell_1,\ldots,\ell_{n-1})$. Since $y_0\not\in Z(g)$, we have in particular that 
the lines $\ell_1,\ell_2,\ell_3$ of $L(y_0)$ are neither all concurrent nor all coplanar.
By Lemma~\ref{3lines}, we have  $\dim(\pi^{-1}(y_0)\cap X_G)=1$.
Applying Lemma~\ref{pre:fib} once again, we get 
$$
\dim X_G\le \dim X_{G'_{12}}+\dim (\pi^{-1}(y_0)\cap X_G)\le 2n+2+1=2n+3.
$$
This completes the proof.
\hfill$\square$

\section{Hendrickson graphs}
In this section we introduce a characterization for graphs $G$, such that $G\subset G_L$ guarantees that the lines of $L$ are either 
all concurrent or all coplanar. As already discussed in the introduction, for this we need to restrict ourselves 
only to {\it generic} configurations $L$,
defined as follows.
\begin{definition}\label{def:generic}
Let $G$ be a graph and put $k:=\dim X_G$ . A point $x\in X_G$ is called {\it generic} if it is a regular point of a $k$-dimensional 
irreducible component of $X_G$.
\end{definition}

The following theorem is the main result of this section, preceded a definition needed for its statement. 
\begin{definition}\label{def:redundant}
A graph $G = (V ,E)$ is called {\it redundant} if, for every edge $e\in E$, 
there exists a subgraph $G'\subset G\setminus\{e\}$ which is Laman.  
A graph $G$ is called a {\it Hendrickson graph} if it is redundant and $3$-(vertex-)connected.\footnote{Recall 
that a graph $G$ is $3$-(vertex-)connected if a removal of any pair of vertices of $G$ results in a connected graph.}
\end{definition}

\begin{theorem}\label{main2}
Let $G$ be a  Hendrickson graph on $n$ vertices, and let $X$ be any $(2n+3)$-dimensional irreducible component of $X_G$.
Then, for every $x\in X$, the lines of $L(x)$ are either all concurrent or all coplanar.
\end{theorem}

For the proof of Theorem~\ref{main2} we use the following constructive characterization
of Hendrickson graphs obtained by Jackson and Jord\'an~\cite{JJ05}.
\begin{theorem}[Jackson and Jord\'an~\cite{JJ05}]\label{constJJ}
Every Hendrickson graph 
can be built up from $K_4$ by a sequence of edge additions and $1$-extensions.
\end{theorem}

\noindent {\it Remark.} Note that applying a $1$-extension or an edge addition preserves 
the property of being Hendrickson.

\subsection{Proof of Theorem~\ref{main2}}
Let $G$ be a Hendrickson graph on the vertex set $[n]$.
By Theorem~\ref{constJJ}, there exists a sequence $G_0,\ldots, G_N$ of Hendrckson (i.e., $3$-connected and redundant) graphs, 
such that $G_0=K_4$, $G_N=G$, and $G_i$ is obtained 
from $G_{i-1}$ by an edge addition or a $1$-extension,
for each $i=1,\ldots, N$.

We use induction on $N$. 
The base case $N=0$ and $G_0=K_4$ is trivial, because a simple case analysis shows that 
every four distinct lines in $\C^3$ which pairwise intersect are either all concurrent or all coplanar.
Assume the correctness of the statement for $0\le N'<N$, and assume that
$G$ is obtained from $G':=G_{N-1}$ by either adding a new edge or by applying a $1$-extension to $G'$;
by the remark following Theorem~\ref{constJJ}, $G'$ is $3$-connected and redundant.

Consider first the case where $G$ is obtained from $G'$ by adding a new edge, say $\{1,2\}$.
By assumption, each of $G,G'$ contains a subgraph which is Laman, and thus 
each of $X_G$ and $X_{G'}$ is $(2n+3)$-dimensional, by Theorem~\ref{main}.
Clearly, $G'\subset G$ and so $X_G\subset X_{G'}$. 
In particular, every $(2n+3)$-dimensional irreducible component $X$ of $X_G$
is also a component of $X_{G'}$. By the induction hypothesis, for every regular point $x\in X$, 
the sequence $L(x)$ consists of either $n$ concurrent or $n$ coplanar lines. 
This completes the proof for this case.

Consider next the case where $G$ is obtained from $G'$ by a $1$-extension. Up to renaming the vertices, 
we may assume that $G'$ is a graph on the vertex set $[n-1]$ and $G$ is obtained 
from $G'$ by adding a new vertex $n$, and by replacing an edge $\{1,2\}$ by three edges 
$\{1,n\}$, $\{2,n\}$, $\{3,n\}$.
Let $G'_{12}$ be the graph obtained from $G'$ by removing its edge $\{1,2\}$.
Let $\pi:\C^{4n}\to \C^{4n-4}$ stand for the projection of $\C^{4n}$ onto its first $4n-4$ coordinates.
As argued  in the proof of Theorem~\ref{main}, we have
$X_{G'_{12}}=\cl (\pi(X_G))$.

Our assumption that $G'$ is redundant implies that $G'_{12}$ contains a subgraph which is Laman. 
Hence, each of $X_{G'_{12}}$, $X_{G'}$ is of dimension $2(n-1)+3=2n+1$, by Theorem~\ref{main}.
Note that $G$ is obtained from $G'_{12}$ by adding a new vertex, $n$, of degree $3$.

Let $X$ be a $(2n+3)$-dimensional irreducible component of $X_G$ and let 
$x$ be a regular point of $X\cap \hat X_G$ (note that this intersection is nonempty since $X\subset \cl(\hat X_G)$).
Let $X'$ be some irreducible component of $\cl(\pi(X))$ that contains $\pi(x)$ (if $\pi(x)$ lies on
more than one such irreducible component, choose $X'$ to be one with maximal dimension).

We claim that $X'$ is $(2n+1)$-dimensional. Indeed, for contradiction, assume that $X'$ is of dimension at most $2n$
($X'$ cannot have higher dimension, because $X_{G_{12}'}$ is $(2n+1)$-dimensional).
Note that, for every $y\in X'$, $\pi^{-1}(y)\cap X$ is at most two-dimensional, by Lemma~\ref{3lines}. 
Combined with Lemma~\ref{pre:fib}, we get
$$
\dim X\le \dim X'+\dim (\pi^{-1}(y)\cap X)\le 2n+2,
$$ 
which contradicts our assumption that $X$ is $(2n+3)$-dimensional.

We next claim that $X'$ is also an irreducible component of $X_{G'}$ (and not only of  $X_{G'_{12}}$);
that is, we claim that $X'\subset X_{G'}$.
Indeed, assume for contradiction that $X'\cap X_{G'}\neq X'$ 
(and hence, by Lemma~\ref{pre:dim}, the intersection is of dimension at most $\dim X'-1=2n$).
By the definition of $X_{G'_{12}}$, for every point $y\in X'\minus X_{G'}$, the first eight coordinates 
(which represent the first two lines in the sequence $L(y)$)
correspond to a pair of non-intersecting lines. 
In particular, $\pi^{-1}(y)\cap X$ is one-dimensional, by Lemma~\ref{3lines}.
Hence, using Lemma~\ref{pre:fib}, we have
$$
\dim X\le \dim X'+\dim (\pi^{-1}(y)\cap X)\le 2n+1+1=2n+2,
$$
which contradicts our assumption that $X$ is $(2n+3)$-dimensional.
Hence $X'\subset X_{G'}$, as claimed.

So $X'$ is a $(2n+1)$-dimensional irreducible component of $X_{G'}$. By the induction hypothesis, 
every point of $X'$ corresponds to a sequence of $n-1$ lines which are either all concurrent or all coplanar.
In particular, for the point $y:=\pi(x)$, we have that the $n-1$ lines of the sequence
$L(y)=(\ell_1,\ldots,\ell_{n-1})$ are either all concurrent or all coplanar. 
Since $x\in\pi^{-1}(y)\cap X$, we have that
$L(x)=(\ell_1,\ldots,\ell_{n-1},\ell_n)$, where the line $\ell_n$ intersects $\ell_1,\ell_2,\ell_3$.

Assume first that $(\ell_1,\ldots,\ell_{n-1})$ are all concurrent. If $\ell_n$ passes through the common intersection point
of $\ell_1,\ell_2,\ell_3$, which is the same intersection point  of $\ell_1,\ldots,\ell_{n-1}$, then the $n$ lines
of $L(x)$ are all concurrent. Otherwise, $\ell_1,\ell_2,\ell_3$ must lie on a common plane with $\ell_n$, and so
$\ell_1,\ell_2,\ell_3$ are both concurrent and coplanar. 
Moreover, by continuity, this is the case for every point $\xi$ in a sufficiently small open 
neighborhood of $x\in X$. Hence the {\it local dimension} 
of $x\in X$ is at most $2n+2$.
Indeed, the configurations involve 
$n-1$ concurrent lines where three of them are coplanar, and an $n$th line, coplanar with the first three lines.
To specify such a configuration, we need three parameters to specify the point $o$ of concurrency,
$2(n-2)$ parameters to specify $n-2$ of the lines through $o$, except for $\ell_3$, one to specify $\ell_3$,
and two to specify $\ell_n$, for a total of $3+2(n-2)+1+2=2n+2$.
Since $x$ is a regular point of $X$, its local dimension is (well defined and) equals to $\dim X=2n+3>2n+2$. 
This contradiction implies that the $n$ lines of $L(x)$ are all concurrent in this case.

Similarly, assume that $(\ell_1,\ldots,\ell_{n-1})$ are all coplanar. If $\ell_n$ lies on the plane $h$ that supports 
$\ell_1,\ell_2,\ell_3$, which is the same plane that supports $\ell_1,\ldots,\ell_{n-1}$, then the $n$ lines
of $L(x)$ are all coplanar. Otherwise, $\ell_1,\ell_2,\ell_3,\ell_n$ must be  concurrent, where
their common point is the unique intersection point of $\ell_n$ with $h$.
So $\ell_1,\ell_2,\ell_3$ are both concurrent and coplanar. 
Moreover, by continuity, this is the case for every point $\xi$ in a sufficiently small open 
neighborhood of $x\in X$. But then the local dimension of $x\in X$ is at most $2n+2$
(here we have $n-1$ coplanar lines where three of them are also concurrent, 
and an $n$th line concurrent with the first three lines, and the analysis is symmetric to the one given above).
Since $x$ is a regular point of $X$, this yields a contradiction, as above. 
Thus the $n$ lines of $L(x)$ are all coplanar in this case.

To recap, we have shown so far that in case $x\in X$ is a regular point, 
$L(x)$ is a sequence of $n$ lines that are either 
all concurrent or all coplanar. By continuity, this is the case for every point of $X$. 
This establishes the induction step, and thus completes the proof. 
\hfill$\square$


\section{Rigidity of planar embeddings of graphs}\label{sec:rigid}

\subsection{Definitions and basic properties}
In this section we introduce some basic definitions of the classical notion of 
combinatorial rigidity of graphs, focusing on planar embeddings. 
For more details about the notions being reviewed here, see  \cite{AR1,Con05} and references therein.

Let $G=(V,E)$ be a graph with $n$ vertices and $m$ edges, 
and write $(v_1,\ldots,v_n)$ and $(e_1,\ldots, e_m)$ for the vertices and edges of $G$, respectively. 
Let $\p:V\mapsto\reals^2$ be an injection, referred to as a (planar) \emph{embedding} of $G$.
We often identify an embedding $\p$ with a point in $\R^{2n}$, in the obvious way.
With this identifications, we define $f_{G}:\R^{2n}\to\reals^{m}$ by 
$$
f_G(\p)=(\ldots,\|\p(v_i)-\p(v_j)\|^2,\ldots)\in\reals^m,
$$
for each point $\p\in \R^{2n}$, where the entries correspond to the edges $(v_i,v_j)\in E$
in their prescribed order, and where $\|\cdot\|$ denotes the Euclidean norm in $\reals^2$. 
Note that $f_G$ is well defined even for points $\p\in \R^{2n}$ that correspond to embeddings which are not injective.
We refer to $f_G$ as the {\it edge function} of the graph $G$.

For a graph $G$, with $n$ vertices and $m$ edges, let $r(G):= \max\{{\rm rank}J_{f_G}(\p)\mid \p\in \R^{2n}\}$,
where $J_{f_G}$ stands for the $m\times (2n-3)$ Jacobian matrix of $f_G$.
Then $G$ is called {\it rigid} if $r(G)=2n-3$, and {\it flexible}, otherwise (see  Asimow and Roth~\cite{AR1} for equivalent definitions and for more details).
Equivalently, $G$ is rigid if and only if the image of $f_G$ over
$\R^{2n}$ forms a $(2n-3)$-dimensional algebraic variety in $\R^m$.
Note that a graph with $n$ vertices and $m<2n-3$ edges is never rigid.
A graph $G$ is called {\it minimally rigid} if it is rigid and has exactly $m=2n-3$ edges. 

We have the following result from rigidity theory.
\begin{lemma}\label{rigidlaman}
If $G$ is rigid, then there exists a subgraph $G'\subset G$ which is Laman.
\end{lemma}
\noindent{\it Remark.} The other direction of Lemma~\ref{rigidlaman} is true too (\cite{Cra}), but we prove it independently (see Theorem~\ref{mainnec}).

We say that a point $\p\in \R^{2n}$ is a {\it regular} embedding of $G$
if ${\rm rank }J_{f_G}(\p)=r(G)$, and {\it singular}, otherwise.
We say that a point $\p\in \R^{2n}$ is a {\it generic} embedding of $G$ 
if $\p$ is regular and ${\bf y}:=f_G(\p)$ is a regular point of the variety $I:=f_G(\R^{2n})$.
A graph $G$ is called {\it globally rigid} if, for every pair $\p,\p'\in \R^{2n}$ 
of generic embeddings of $G$ such that $f_G(\p)=f_G(\p')$, the sets
$\p(V)$ and $\p'(V)$ are congruent.\\

\noindent{\it Remark.} We note that the standard definition (see, e.g., Connelly~\cite{Con05}) of global rigidity refers only to embeddings $\p$
of a graph $G$ which are ``generic'' in the sense that their coordinates are algebraically independent over $\mathbb Q$. 
In our definition we consider a larger set of embeddings of $G$ and consider them as generic.
As we will see, both definitions yield the same family of globally rigid graphs. So our definition 
is better, in the sense that it applies to a larger set of embeddings of $G$.

We have the following result from rigidity theory.\footnote{Note 
that Lemma~\ref{uniquehend} applies also to our notion of global rigidity, since (a priori) our notion is more restrictive.} 
\begin{lemma}[Hendrickson~\cite{Hen88}]\label{uniquehend}
If $G$ is globally rigid, then $G$ is Hendrickson.
\end{lemma}
The other direction, namely, that every Hendrickson graph $G$ is globally
rigid, follows by combining the two results of Connelly~\cite{Con05} and of
Jackson and Jordan \cite{JJ05}.
The analysis in this paper reproves this fact (extending it slightly, by showing it applies to 
a larger set of ``generic" embeddings), using only \cite{JJ05},
and bypassing, or finding alternative proof, for the result from \cite{Con05}.

\subsection{Pairs of embeddings that induce the same edge distances}
We consider the following (real) variety 
$$
V_G:=\{ (\p,\p')\in(\R^{2n})^2\mid f_G(\p)=f_G(\p')\}. 
$$
Note that $\dim_\R V_G\ge 2n+3$, for every graph $G$ on $n$ vertices, simply since it contains the subvariety
$$\{ (\p,\p')\in (\R^{2n})^2\mid \text{$\p'$ is congruent to $\p$}\},$$
and the latter has (real) dimension $2n+3$, as can be easily verified.

We have the following property.
\begin{lemma}\label{rigidpairs}
Let $G$ be a graph on $n$ vertices. If $\dim_\R V_G=2n+3$,
then $G$ is rigid.
\end{lemma}
\noindent{\it Remark.} Later (Theorem~\ref{mainnec}) we show that in fact $\dim_\R V_G=2n+3$ if and only if $G$ is rigid.
\begin{proof}
Assume that $G$ is flexible. 
Let $\pi:\R^{4n}\to\R^{2n}$ denote the projection of $\R^{4n}$ onto its first $2n$ coordinates.
Clearly, $\pi(V_G)=\R^{2n}$. We show that, for every $\p\in\R^{2n}$, the preimage $\pi^{-1}(\p)\cap V_G$
is at least four-dimensional. 
As above, for an embedding $\p\in\R^{2n}$ of $G$, define the variety
$$T(\p):=\{\p'\in\R^{2n}\mid \p'~\text{is congruent to $\p$}\}.$$ 
As was already noted, for $\p$ fixed, $T(\p)$ is $3$-dimensional.

Let $\p_0$ be a regular embedding of $G$. By our assumption that $G$ is flexible, 
there exists a continuous path $\q(t)$, $t\in[0,1)$, such that $\q(0)=\p_0$, $\q(t)\in f_G^{-1}(f_G(\p_0))$, and $\q(t)$ is 
not congruent to $\p_0$, for every $t\in(0,1)$ (see Asimow and Roth~\cite[Proposition 1]{AR1}).
Then the set 
$$
\{\q\mid t\in(0,1), \q~\text{is congruent to}~\q(t))\}=\bigcup_{t\in(0,1)}T(\q(t))
$$
is contained in $\pi^{-1}(\p)\cap V_G$ and forms a four-dimensional real manifold, as is not hard to verify
(note that the sets in the union are pairwise disjoint).
Thus, $\dim_\R\pi^{-1}(\p)\cap V_G\ge 4$ for every regular embedding $\p$ of $G$, which implies
that $\dim_\R V_G\ge 2n+4$.
This completes the proof of the lemma.
\end{proof}

\begin{lemma}\label{globalpairs}
Let $G$ be a graph on $n$ vertices. 
Suppose that every irreducible component of $V_G$ which has maximal dimension is contained in 
$$\{ (\p,\p')\in (\R^{2n})^2\mid \text{$\p'$ is congruent to $\p$}\}.$$
Then $G$ is globally rigid.
\end{lemma}
\begin{proof}
By assumption $\dim V_G= 2n+3$, and so, by Lemma~\ref{rigidpairs},
$G$ is rigid.

Let $\p_0,\p_0'$ be generic embeddings of $G$ such that $f_G(\p_0)=f_G(\p_0')$.
We claim that the pair $(\p_0,\p_0')$ lies on an irreducible component of $V_G$ of 
maximal dimension (that is, of dimension $2n+3$).
Put ${\bf y}_0:=f_G(\p_0)=f_G(\p_0')$ and $I:=f_G(\R^{2n})$. 
Let $N,N'\subset\R^{2n}$ be some open neighborhood of $\p_0,\p_0'$,  respectively.
Since each of $\p_0$ (resp., $\p_0'$) is generic, taking $N$ (resp., $N'$) to be sufficiently small, 
we may assume that the image of ${f_{G}}_{|_N}$ (resp., ${f_{G}}_{|_{N'}}$),  
the restriction of $f_G$ to $N$ (resp., to $N'$), is $(2n-3)$-dimensional.
Moreover, we may assume that $f_G(N)=f_G(N')$.
Indeed, each of $f_G(N),f_G(N')$ is a (relatively) open neighborhood of ${\bf y}_0$ in $I$, so their intersection (since it is nonempty)
must be open. 

Put $M:=f_G(N)=f_G(N')$. By what have just been argued, $M$ is a $(2n-3)$-dimensional  neighborhood of ${\bf y}_0$ in $I$,
and, for every ${\bf y}\in M$, we have $(\p,\p')\in V_G\cap (N\times N')$, where 
$\p\in{f_{G}}_{|_N}^{-1}({\bf y})$, $\p'\in {f_{G}}_{|_{N'}}^{-1}({\bf y})$, and each of 
${f_{G}}_{|_N}^{-1}({\bf y})\cap N$ and 
${f_{G}}_{|_{N'}}^{-1}({\bf y})\cap N'$ is 3-dimensional.
In other words, $V_G\cap (N\times N')$ is 
a neighborhood of $(\p_0,\p_0')$ in $V_G$ which is 
of dimension $2n-3+6=2n+3$.
So $(\p_0,\p_0')$ necessarily lies on an irreducible component of $V_G$ of maximal dimension.
This establishes the claim. 

Our assumption then implies that $\p_0$ is congruent to $\p_0'$.
Since this is true for every pair $(\p_0,\p_0')$ of generic embeddings of $G$, the lemma follows.
\end{proof}


\section{Reduction to line incidences in three dimensions}\label{reduction}

We apply the Elekes--Sharir framework (see~\cite{ES,GK2}) to connect the notion of graph rigidity of planar structures,
discussed in Section~\ref{sec:rigid}, with line configurations in $\R^3$ (or, rather, in
the real projective 3-space\footnote{To
simplify the presentation, we continue to work in the affine $\R^3$, 
but the extension of the analysis to the projective setup is straightforward.
Issues related to this extension will be noted throughout the analysis.}).
Specifically, we represent each orientation-preserving rigid motion of the plane 
(called a \emph{rotation} in \cite{ES,GK2}) as a point
$(c,\cot (\theta/2))$ in $\R^3$, where $c$ is the center of rotation, 
and $\theta$ is the (counterclockwise) angle of rotation. (Note that pure translations are mapped
in this manner to points at infinity.) Given a pair of distinct points $a,b\in\R^2$,
the locus of all rotations that map $a$ to $b$ is a line $\ell_{a,b}$ in the above
parametric 3-space, given by the parametric equation 
\begin{equation} \label{line}
\ell_{a,b} = \{ \left( u_{a,b} + tv_{a,b},\;t\right) \mid t\in\R \} ,
\end{equation}
where $u_{a,b} = \tfrac12(a+b)$ is the midpoint of $ab$, and $v_{a,b}=\tfrac12(a-b)^{\perp}$ is a vector
orthogonal to $\vec{ab}$ of length $\tfrac12\|a-b\|$, with $\vec{ab}$, $v_{a,b}$ positively oriented
(i.e., $v_{a,b}$ is obtained by turning $\vec{ab}$ counterclockwise by $\pi/2$). 

It is instructive to note (and easy to verify) that every non-horizontal line $\ell$ in $\R^3$ can be written 
as $\ell_{a,b}$, for a unique (ordered) pair $a,b\in\R^2$. 
More precisely, if $\ell$ is also non-vertical, the resulting $a$ and $b$ 
are distinct. If $\ell$ is vertical, then $a$ and $b$ coincide, at the intersection of $\ell$ with the $xy$-plane, 
and $\ell$ represents all rotations of the plane about this point.

A simple yet crucial property of this transformation is that, for any pair of pairs $(a,b)$
and $(c,d)$ of points in the plane, $\|a-c\|=\|b-d\|$ if and only if $\ell_{a,b}$ and $\ell_{c,d}$
intersect, at (the point representing) the unique rotation $\tau$ that maps $a$ to $b$
and $c$ to $d$. This also includes the special case where $\ell_{a,b}$ and $\ell_{c,d}$ are
parallel, corresponding to the situation where the transformation that maps $a$ to $b$
and $c$ to $d$ is a pure translation (this is the case when $\vec{ac}$ and $\vec{bd}$ are parallel (and of equal length)).

Note that no pair of lines $\ell_{a,b}$, $\ell_{a,c}$ with $b\neq c$ can intersect (or be parallel), 
because such an intersection would represent a rotation that maps $a$ both to $b$ and to $c$,
which is impossible.

\begin{lemma} \label{cong}
Let $L = \{\ell_{a_i,b_i} \mid a_i, b_i \in \R^2,\; i=1,\ldots,r\}$ be a collection of $r\ge 3$ (non-horizontal) lines in $\R^3$.

\noindent
(a) If all the lines of $L$ are concurrent, at some common point $\tau$,
then the sequences $A=(a_1,\ldots,a_r)$ and $B=(b_1,\ldots,b_r)$ are congruent, 
with equal orientations, and $\tau$ (corresponds to a rotation that) maps $a_i$ to $b_i$, for each $i=1,\ldots,r$.

\noindent
(b) If all the lines of $L$ are coplanar, within some common plane $h$,
then the sequences $A=(a_1,\ldots,a_r)$ and $B=(b_1,\ldots,b_r)$ are congruent, 
with opposite orientations, and $h$ defines, in a unique manner, an orientation-reversing 
rigid motion $h^*$ that maps $a_i$ to $b_i$, for each $i=1,\ldots,r$.

\noindent
(c) If all the lines of $L$ are both concurrent and coplanar, then the points
of $A$ are collinear, the points of $B$ are collinear, and $A$ and $B$ are congruent.
\end{lemma}

\noindent{\bf Proof.}
(a) This is an immediate consequence of the properties of the Elekes--Sharir transformation,
as noted above.

\noindent
(b) For each pair of indices $i\ne j$, the lines $\ell_{a_i,b_i}$ and $\ell_{a_j,b_j}$
intersect or are parallel. Hence $\|a_i-a_j\|=\|b_i-b_j\|$. This implies that the sequences
$A$ and $B$ are congruent. Assume that the lines $\ell_{a_i,b_i}$ are not all concurrent
(this is the case that will be addressed in part (c)). That is, there is no orientation-preserving 
rigid motion that maps $A$ to $B$, so $A$ and $B$ must have opposite orientations, and the unique rigid 
motion $h^*$ that maps $A$ to $B$ is orientation-reversing.

\noindent
(c) As in the previous cases, $\|a_i-a_j\|=\|b_i-b_j\|$ for every $i,j$, and hence the sequences $A$ and $B$ are necessarily congruent.
As in part (a), since the lines of $L$ are concurrent, there exists an orientation-preserving rigid motion $\tau$
that maps $A$ to $B$. In our case the lines of $L$ also lie on a common plane $h$, and thus any line $\lambda$ in $h$
intersects each of the lines of $L$. Choose $\lambda$ in $h$ so that the lines of $L\cup\{\lambda\}$ are not all concurrent. 
Note that we have $\lambda=\ell_{a',b'}$, for some $a',b'\in \R^2$, as every line in $\R^3$ can be interpreted in this way.
By part (b), applied  to the set $L\cup\{\lambda\}$, there exists an orientation-reversing rigid 
motion $h^*$ that (in particular) maps the 
$A$ to $B$. So $B$ is the image of $A$ (as sequences) under
an orientation-preserving rigid motion, and also under an orientation-reversing rigid motion, which can
happen only if both sets are collinear. 
$\hfill\Box$


\section{Necessity of our conditions}\label{sec:nec}

In this section we show that the conditions in Theorem~\ref{main} and Theorem~\ref{main2}
are not only sufficient, but also necessary. That is, we have the following statements.
\begin{theorem}\label{mainnec}
Let $G$ be a graph on $n$ vertices.
Then $X_G$ is $(2n+3)$-dimensional if and only if there exists a subgraph $G'\subset G$ which is Laman.
\end{theorem}

\begin{theorem}\label{main2nec}
Let $G$ be a graph on $n$ vertices.
Then $G$ is Hendrickson if and only if  every irreducible component of maximal dimension $X\subset X_G$ is contained in $X_{K_n}$.
\end{theorem}

The following observation says that $X_G\cap \R^{4n}$ contains an isomorphic copy of $V_G$.
\begin{lemma}\label{WX}
Let $G=(V,E)$ be a graph on $n$ vertices and $m$ edges. Then there exists a polynomial mapping $\varphi:(\R^{2n})^2\to \R^{4n}$
such that $\varphi(V_G)\subset X_G\cap \R^{4n}$. 
\end{lemma}
\begin{proof}
Define $\varphi:(\R^{2n})^2\to \R^{4n}$ by
$$((p_1,\ldots,p_n),(p'_1,\ldots,p'_n))\mapsto (\ell_{p_1,p_1'},\ldots,\ell_{p_n,p_n'}).
$$
We claim that $\varphi(V_G)\subset X_G\cap \R^{4n}$.
Indeed, by the definition of $V_G$, $\|p_i-p_j\|=\|p'_i-p'_j\|$, for every $(i,j)\in E$.
By the properties reviewed in Section~\ref{reduction}, the lines $\ell_{p_i,p_i'}$ and $\ell_{p_j,p_j'}$ 
then must intersect, for every $(i,j)\in E$,
and thus $(\ell_{p_1,p_1'},\ldots,\ell_{p_n,p_n'})\in X_G$. Since the points $p_i,p_i'$, for $i=1,\ldots,n$, have 
real coordinates, the representation of each of the lines $\ell_{p_i,p_i'}$, for $i=1,\ldots,n$,
 as points in $\C^4$, requires only real coefficients.
Thus $(\ell_{p_1,p_1'},\ldots,\ell_{p_n,p_n'})\in X_G\cap \R^{4n}$, as claimed.
\end{proof}

We are now ready to prove Theorem~\ref{mainnec} and Theorem~\ref{main2nec}.
\begin{proof}[Proof of Theorem~\ref{mainnec}]
Let $G$ be a graph on $n$ vertices.
If $G$ contains a subgraph $G'$ which is Laman, then $X_G$ is $(2n+3)$-dimensional, by Theorem~\ref{main}
(and the remark following it).

We now prove the opposite direction. For contradiction, assume that $X_G$ is $(2n+3)$-dimensional, and that
$G$ does not contain a subgraph which is Laman. By Lemma~\ref{rigidlaman}, $G$ is flexible.
Using Lemma~\ref{rigidpairs} (and the fact that $\dim V_G\ge 2n+3$, for every graph $G$), we get that 
$\dim V_G\ge 2n+4$.
Lemma~\ref{WX} then implies that also $\dim_{\C}X_G\ge \dim_{\R} X_G\ge 2n+4$.
This contradicts our assumption about the dimension of $X_G$, 
and by this completes the proof.
\end{proof}

\begin{proof}[Proof of Theorem~\ref{main2nec}]
Let $G$ be a graph on $n$ vertices.
If $G$ is Hendrickson, then every irreducible component of maximal dimension 
$X\subset X_G$ is contained in $X_{K_n}$, by Theorem~\ref{main2}.

For the opposite direction, assume that
every irreducible component of maximal dimension $X\subset X_G$ is contained in $X_{K_n}$.
Since 
$$
X_{K_n}\cap \R^{4n} \subset \varphi(V_G)\subset X_G\cap\R^{4n},
$$ 
it follows that irreducible components of $V_G$ which have maximal dimension
are exactly $X_{K_n}\cap \R^{4n}$.
Thus $G$ is globally rigid, by Lemma~\ref{globalpairs}.
Using Lemma~\ref{uniquehend},  we get that $G$ is Hendrickson, as asserted.
\end{proof}

\section{Applications}
\subsection{Standard rigidity}
We reprove the following theorem of Connelly~\cite{Con05} (see also \cite{JJS}
for a simplification of the proof in \cite{Con05}).
\begin{theorem}
If $G$ is obtained from $K_4$ by a sequence
of edge additions and $1$-extensions, then $G$ is globally rigid in $\R^2$.
\end{theorem}
\begin{proof}
This follows by combining Theorem~\ref{main2}, Lemma~\ref{globalpairs} and Lemma~\ref{WX}.
\end{proof}

\subsection{Rigidity and global rigidity on a two-dimensional sphere}
Consider the rigidity problem, where this time the vertices of a graph $G$ are embedded to 
the unit sphere in $\R^3$. Note that isometries of the sphere can be represented as points in $\R^3$ (see \cite{Tao}).
Thus, in view of our general results about lines, we obtain the following corollaries.
\begin{corollary}
Consider embeddings of graphs in the unit sphere $\S^2\subset\R^3$.
Then a graph $G$ is rigid if and only if there exists a subgraph $G'\subset G$ which is Laman.
\end{corollary}

\begin{corollary}
Consider embeddings of graphs in the unit sphere $\S^2\subset\R^3$.
Then a graph $G$ is globally rigid if and only if there exists a subgraph $G$ is Hendrickson.
\end{corollary}

\appendix 
\section{Tools from algebraic geometry}
\label{sec:alggeom}

We use the following basic facts from algebraic geometry theory. 
\begin{lemma}[{\cite[Proposition I.7.1]{Ha77}}]
\label{pre:dim}
Let $X\subset \C^D$ be an irreducible variety of dimension $k$ and $f\in \C[z_1,\ldots,z_D]$.
Then $X\cap Z(f)$ is either $X$, the empty set, or has 
dimension $k-1$.
\end{lemma}

\begin{lemma}\label{pre:fib}
Let $\pi$ denote the projection of $\C^D$ onto its first $k$ coordinates and let $X$ be an irreducible subvariety of $\C^D$.
Let $Y$ denote the Zariski closure of $\pi(X)$. Then
$$\dim X\le \dim Y+\dim (\pi^{-1}(y)\cap X)$$
for every $y\in \pi(X)$.
\end{lemma}

\begin{lemma}\label{pre:proj}
Let $\pi$ denote the projection of $\C^D$ onto its first $k$ coordinates and let $S\subset \C^d$ be any subset.
Then
$
\cl(\pi(\cl(S)))=\cl(\pi(S)).
$
\end{lemma}

\end{document}